\documentclass[11pt]{amsart}
\usepackage{amsfonts,amssymb,amscd,amsmath,enumerate,url,verbatim}
  \usepackage[dvips]{epsfig}
 \usepackage{amsgen, amstext,amsbsy,amsopn, amsthm, amsfonts,amssymb,amscd,amsmat
 h,euscript,enumerate,url,verbatim,calc,xypic}
 \usepackage[colorlinks]{hyperref}
 \usepackage{latexsym}
 \usepackage{graphics}
 \usepackage{color}

 \newcommand{\n}{\mathfrak{n} }
 \newcommand{\m}{\mathfrak{m} }

    \newcommand{\kernel}{\operatorname{Ker}}

 \newcommand{\gr }{\operatorname{gr}}

 \newcommand{\Tor}{\operatorname{Tor}}

 \theoremstyle{plain}
 \newtheorem{thm}{Theorem}[section]
 \newtheorem{prop}[thm]{Proposition}
 \newtheorem{lem}[thm]{Lemma}
 \newtheorem{cor}[thm]{Corollary}
 
\theoremstyle{definition}
 
 \newtheorem{exam}[thm]{Example}

 \newtheorem{ques}[thm]{Question}
\theoremstyle{remark}
 \newtheorem{rem}[thm]{Remark}

\begin{document}

\title[ ] { The Golod property for powers of ideals  and  Koszul ideals }

\begin{abstract}
Let $S$ be a regular local ring or a polynomial ring over  a field
and $I$ be an ideal of $S$. Motivated by a recent result of Herzog
and Huneke, we study the natural question of whether $I^m$ is a
Golod ideal for all $m\geq 2$. We observe that the Golod property of
an ideal can be detected through the vanishing of certain maps
induced in homology. This observation leads us to generalize some
known results from the graded case to local rings  and obtain new
classes of Golod ideals.

 \end{abstract}

 \author[R.~ Ahangari Maleki ]{Rasoul Ahangari Maleki}
\address{Department of Mathematics, Alzahra University, Vanak, Tehran,
Iran, Zip Code 19834; and School of Mathematics, Institute for
Research in Fundamental Sciences (IPM), P.O. Box: 19395-5746,
Tehran, Iran } \email{rahangari@ipm.ir}

%\numberwithin{equation}{thm}

%%\begin{abstract}

%\thispagestyle{empty}

%\section*{Introduction}

\thanks{This work  was  jointly supported by the Iran National Science Foundation (INSF) and Alzahra University grant No. 95001343. This research was also in part supported by a grant from IPM (No. 95130111). }

 \maketitle
\section{Introduction }
Throughout this paper  we let $(S,\n,k)$ denote a regular  local
ring with maximal ideal $\n$ and residue field $k$ or a polynomial
ring over a field $k$ with graded maximal ideal $\n$. All modules
are assumed to be graded if the ring is graded. Let $I$ be an ideal of $S$.  The Poincar\'{e} series of a finitely
generated $R=S/I$-module $M$  is denoted by $P_{M}^{R}(t)$ and is
defined to be formal power series $\sum_{i\geq
0}\dim_k\Tor_{i}^{R}(M,k)t^i$. In general, this power series is not
a rational function. We refer the reader for the history of the
rationality of Poincar\'{e} series to the survey article \cite{Av}
by Avramov. On the other hand, Serre \cite{Ser} showed that  there
is a coefficientwise inequality of formal power series
\begin{equation}\label{Serre}
P_k^{R}(t)\leq\frac{(1+t)^d}{1-t\sum_{i>0} \dim_k H_i(\mathcal{K})t^i},
\end{equation}

where $d$ is the embedding dimension of $R$ and $\mathcal{K}$ is the
Koszul complex of $R$ with respect to a minimal system of generators
of its maximal ideal.

We say that the ring $R$ or the ideal $I$ is Golod if $P_k^{R}(t)$
coincides with the upper bound given by Serre.
 Golod rings are an
example of good rings in the sense that all finitely generated
modules over such rings have rational Poincar\'{e} series sharing a
common denominator, see \cite{Av-K}. Many results regarding classes
of Golod ideals are established in the case of graded rings. If $S$
is a polynomial ring over a field of characteristic zero, Herzog and
Huneke in \cite{HH} identified large classes of Golod ideals. They
showed, among other results, that the powers $I^m$ of an ideal $I$
are Golod for all $m\geq 2$.

The main goal of this paper is to study the Golod property of ideals
of a regular local ring. In view of  results of Herzog and Huneke,
it is a natural question to ask whether  the
 results of  \cite{HH} hold if $S$ is a polynomial ring over a field of
 arbitrary characteristic or more generally if $S$ is a regular local
ring rather than a polynomial ring. A known fact in this direction
is a result of Herzog et al. \cite{HWY} which says that large powers
of an ideal are Golod. Another  evidence in support of the question
is that if $I$ is a complete intersection, then  $I^m$ is Golod for
all $m\geq 2$, see \cite{Av2} and \cite{Gv}.  Herzog and the author
\cite{HA} recently   removed the assumption on characteristic of $k$
of the results of  \cite{HH} for monomial ideals. The methods used
in the proofs of the results we mentioned above vary from one case
to another, yet the nature of the results suggests that there is
some common ground among them. This observation allows us to
partially generalize to local rings some of the known results for
graded rings.

Let $\rho(I)$ denote the smallest number $m$ such that for all $r>m$
and all $i>0$ the natural map $ \Tor^S_i(S/I^r,k)\rightarrow
\Tor^S_i(S/I^{r-1},k)$ is the zero map. In Section 2 we show that
the invariant $\rho(I)$ is finite and the following holds:

\begin{thm}
Let $m$ be a positive integer and  $J$ be an ideal of $S$. Then the
following hold:
\begin{enumerate}
\item If  $m>\rho(I)$  and
$I^{2m-2}\subseteq J \subseteq I^m$, then  $J$ is Golod. In
particular, $I^m$ is Golod.

\item  If $d= \dim S$ and
$m\geq\max\{2 d,\rho(I)+d\}$. Then $\overline{I^{m}}$, the integral
closure of $I^{m}$, is Golod.
\end{enumerate}
\end{thm}
If $\rho(I)=1$, then it follows at once from this theorem that $I^m$
is Golod for all $m\geq 2$. This suggests the question
\begin{ques}\label{intq}
Is it true that $\rho(I)=1$ for any ideal $I$ of $S$?
\end{ques}
We show that  Question  \ref{intq} has an affirmative answer
provided that either $S$ is a polynomial ring over a field of
characteristic zero or $S$ has Krull dimension at most $2$, see
Proposition \ref{graded} and Theorem \ref{local}. The result in the
graded case is an immediate consequence of the work of Herzog and
Huneke. When $S$ is local, we prove in  Proposition \ref{ci} that
$\rho(I)=1$ if $I$ is a complete intersection ideal. A similar
conclusion for a monomial ideal $I$, without any assumption on
characteristic, holds true, see \cite{HA}.

Section $3$ of this paper is devoted to study the Golod property of
Koszul ideals. We say that an ideal $I$ of $S$ is Koszul if its
associated graded module with respect to maximal ideal $\n$  has a
linear resolution over the associated graded ring $\oplus_{i\geq
0}\n^i/\n^{i+1}$. This notion can be considered as a generalization
of the notion of  componentwise linear ideal. A necessary (not
sufficient ) condition for an ideal $I$ to be Koszul is that the
natural map $\upsilon^S _i(I):\Tor^S_i(\n I,k)\rightarrow
\Tor^S_i(I,k)$ is zero for all $i$. In \cite{HRW} it is proved that
any componentwise linear ideal of a polynomial ring $S$ is Golod. We
extend this result by showing that
\begin{thm}
Assume  $\upsilon^S _i(I)=0$ for all $i$.  Then we have the following:
\begin{enumerate}
\item If  $I\subseteq \n^2$, then $I$ is Golod.
\item  If $I^2\subseteq L\subseteq \n I$, then the ideal  $L$ is Golod. In particular, for any proper ideal $ J \supseteq I$ the ideal $IJ$ is Golod.
\end{enumerate}
\end{thm}
The condition $I\subseteq \n^2$ is not necessary if $S$ contains a field, see Theorem \ref{field}.

%%%%%%%%%%%%%%%%%%%%%%%%%%%%%%%%%%%%%%%%%%%%%%%%%%%%%%%%%%%%%%%%%%%%%%%%%%%%%%%%%%%%%%%%%%%%%%%%%%%%%%%%%%%%%%%%%%%
\section{ Massey operations and vanishing of maps of $\Tor$}

We use throughout the notation  $K$ to  stand for the Koszul complex of $S$ with
respect to a minimal system of generators of $\n$. The Koszul complex has a differential graded (DG) algebra structure and resolves $k$ over $S$.
We will denote by
$Z$ and $B$ the subcomplexes of the cycles and boundaries of
$K$ respectively.

 Let $I$ be a proper ideal of $S$. Set $R=S/I$ and let $\mathcal{K}$ denote the Koszul complex of $R$
 with respect to a minimal system of generators of the maximal ideal of $R$.  This complex has a DG algebra structure.
 Golod \cite{G} showed that the equality in (\ref{Serre})
 can be characterized by vanishing of all Massey operations in $\mathcal{K}$. By definition this is the case if  $\mathcal{K}$ admits a
trivial Massey operation, see \cite{Av-K} and  \cite{GU}. That is,
for some homogeneous $k$-basis $\mathcal{B}=\{h_{\lambda}\}_{\lambda
\in \Lambda}$ of $H_{\geq 1}(\mathcal{K}):=\oplus _{i\geq
1}H_i(\mathcal{K})$ there exists a function
$\mu:\bigsqcup_{p=1}^{\infty} \mathcal{B}^p \rightarrow \mathcal{K}$
subject to the following conditions:
 \begin{enumerate}
\item [$(M1)$] $\mu(h_{\lambda})=z_{\lambda}\in Z(\mathcal{K})$ with
$[z_{\lambda}]=h_{\lambda}$;\\
\item [$(M2)$]  $\partial \mu(h_{\lambda_1},\ldots,h_{\lambda_p})=
\displaystyle\sum_{j=1}^{p-1}\overline{\mu(h_{\lambda_1},\ldots,h_{\lambda_j})}\mu(h_{\lambda_{j+1}},\ldots,h_{\lambda_p})$
\ \
for $p\geq 2$;\\
\item [$(M3)$] $\mu(\mathcal{B}^p)\subseteq \m \mathcal{K}$ \ \ for $p\geq 1$.
\end{enumerate}
Here we use $\partial$ for  the differential map of $\mathcal{K}$ and  $|a|$ for the degree of  a (homogeneous)
element $a$ of $\mathcal{K}$. Also  we set $\bar{a}=(-1)^{|a|+1}a$.

Therefore the ring $R=S/I$ or the ideal $I$ is Golod  if and only if
$\mathcal{K}$ admits a trivial Massey operation. Observe that
condition $(M2)$ implies that $\mu(h_{\lambda_1})\mu(h_{\lambda_2})$
is a boundary for all $h_{\lambda_1},h_{\lambda_2}\in\mathcal{B}$.
This means that $H(\mathcal{K})$ has trivial multiplication. As it
is shown in \cite{K} this condition does not suffice to characterize
Golod rings.

To show that  an ideal of $S$ is Golod  we utilize the following simple observation of the construction of a trivial Massey operation.
The idea of this is motivated by \cite[Lemma 1.2]{RS}. We apply
similar technique for the proof.

\begin{lem}\label{cycle}
Let $I$ and $L$ be ideals of $S$  such that $L^2\subseteq I\subseteq
L$. Suppose  that  the map $$\Tor^S_i(S/I,k)\rightarrow
\Tor^S_i(S/L,k),$$  induced by the natural surjection
$S/I\rightarrow S/L$, is zero for all $i\geq 1$. Then the ideal $I$
is Golod.

\end{lem}

\begin{proof}

 To prove that $I$ is Golod  (i.e the ring $R=S/I$ is Golod ) we show that $\mathcal{K}$, the Koszul complex of $R$, admits a
trivial Massey operation. If we  establish that   $\mu$ can
be chosen in such a way that $\mu(h_{\lambda_1})\mu(h_{\lambda_2})=0$ for all
$h_{\lambda_1}, h_{\lambda_2}\in \mathcal{B}$, then by setting $\mu(h_{\lambda_1},\ldots,h_{\lambda_n})=0$ for all $n\geq
2$,  obviously $(M2)$ is satisfied and $R$ is Golod. To this end we
note that  the map $\Tor^S_i(S/I,k)\rightarrow \Tor^S_i(S/L,k)$ can be identified with the map
\begin{equation}\label{map}
H_i(\mathcal{K})\rightarrow H_i(\mathcal{K}/L\mathcal{K})
\end{equation}
induced by the natural projection $\mathcal{K}\rightarrow
\mathcal{K}/L\mathcal{K}$.  For each $h_{\lambda}\in\mathcal{B}$,
the vanishing of  (\ref{map}) gives that $h_{\lambda}$ can be
represented as $h_{\lambda}=[z_{\lambda}]$ for some
 $z_{\lambda}\in Z(\mathcal{K})\cap
L\mathcal{K}$.
 Now we define $\mu(h_{\lambda})=z_{\lambda}$. The assumption
$L^2\subseteq I$ implies that
$\mu(h_{\lambda_1})\mu(h_{\lambda_2})=0$ for any two elements
$h_{\lambda_1}, h_{\lambda_2}\in \mathcal{B}$.
\end{proof}
For the remaining of the paper we will apply Lemma \ref{cycle} to
obtain new classes of Golod ideals.
%%%%%%%%%%%%%%%%%%%%%%%%%%%%%%%%%%%%%%%%%%%%%%%%%%%%%%%%%%%%%%%%%%%%%%%%%%%%%%%%%%%%%%%%%%%%%%%%%%%%%%%%%%%%%%%%%%%%%%%%%%%%%%%%%%%%%%%%%%%%%%%

\section{Golodness of powers of ideals}
In this section we study the Golod property of powers of ideals. We
show that large powers of an ideal and their integral closures are
Golod.  We need the following lemma to prove the main results of
this section.

\begin{lem}\label{key lem}
Let $I$  be an ideal of $S$. Then there exists an integer $m_0$ such that for all $m>m_0$ the natural map
$$\Tor^S_i(S/I^m,k)\rightarrow \Tor^S_i(S/I^{m-1},k)$$  induced by the surjection $S/I^m \rightarrow S/I^{m-1}$, is zero for
all $i>0$.
\end{lem}
\begin{proof}
First note that for an ideal $L$ of $S$ we have  the  natural
isomorphism $\Tor^S_i(S/L,k)\cong\Tor^S_{i-1}(L,k)$ for all  $i>0$.
Also for each $S$-module $M$ the Koszul homology $H_i(M\otimes_S K)$
is functorially isomorphic to $\Tor^S_i(M,k)$.  Thus with these
observations, it suffices to show that there exists an integer $m_0$
such that  the containment $Z_i\cap I^mK_{i}\subseteq I^{m-1}Z_i$
holds for all $m>m_0$ and $i\geq0$.

By the Artin-Rees lemma \cite[Lemma 5.1]{E}, there exists an integer
$m_0$ such that for all $m>m_0$
\begin{equation}\label{inc1}
Z_i\cap I^mK_{i}\subseteq I^{m-m_0}(Z_i\cap I^{m_0}K_{i})
\end{equation}
for all $i\geq 0$. On the other hand since $H_i(I^{m_0}K)$ is a
$k$-vector space  we get the first link in the following chain
\begin{equation}\label{inc2}
I(Z_i\cap I^{m_0}K_{i})\subseteq B_i(I^{m_0}K)\subseteq I^{m_0}Z_i.
\end{equation}
 Multiply (\ref{inc2}) by $I^{m-{m_0}-1}$ and combine the resulting inclusion with (\ref{inc1}) to get $Z_i\cap
I^mK_{i}\subseteq I^{m-1}Z_i$ for all $i\geq0$.
\end{proof}
Lemma \ref{key lem} guarantees that  for any proper ideal
$I\subseteq S$ the number

\[\rho(I):=\min\{m: \Tor^S_i(S/I^r,k)\rightarrow \Tor^S_i(S/I^{r-1},k) \ \ \text{is zero for all $r>m$ and  all $i>0$} \}\]
 exists.

 Herzog et al. \cite{HWY} showed that all higher powers
of an ideal of a regular local ring are Golod. The following theorem
generalizes  this result with a simpler proof.
\begin{thm}\label{Golod}
Let $I$  be an ideal of $S$ and  $m$ be a positive integer with
$m>\rho(I)$. If $J$ is an ideal of  $S$ such that $I^{2m-2}\subseteq
J \subseteq I^m$, then $J$ is Golod. In particular, $I^m$ is Golod.
\end{thm}
\begin{proof}
 Applying  Lemma \ref{cycle} for $L=I^{m-1}$, it suffices to show that the  map $\Tor^S_i(S/J,k)\rightarrow
 \Tor^S_i(S/L,k)$, induced by the surjection
 $S/J\rightarrow S/L$,
 is zero for
all $i> 0$. But this map factors through $\Tor^S_i(S/I^{m},k)$,
and the natural map $$\Tor^S_i(S/I^m,k)\rightarrow
\Tor^S_i(S/I^{m-1},k)$$ is zero as $m>\rho(I)$. This finishes
the proof.
\end{proof}
Let $S$ be  a polynomial ring over a field $k$ of characteristic
zero and $I$ be a graded ideal of $S$. In \cite[Theorem 2.11]{HH}
it is shown that the integral closure $\overline{I^{m}}$  of
$I^{m}$ is Golod for all $m>d$ where   $d$ is the Krull
dimension of $S$. For a monomial ideal $I$ the authors proved that
$\overline{I^{m}}$ is Golod for all $m\geq 2$. Very recently in
\cite{HA} the same result has been proved for monomial ideals with
no assumptions on the characteristic. The next result shows that
over a regular local ring or a polynomial ring $S$, with no
additional assumption on characteristic, the integral closure of
$I^m$ is Golod  provided that $m$ is large enough.

\begin{thm}\label{integral}
Let $I$  be an ideal of $S$ and  $m$ be a positive integer with
$m\geq\max\{2 d,\rho(I)+d\}$. Then $\overline{I^{m}}$  is Golod.
\end{thm}
\begin{proof}
We verify the conditions of Lemma \ref{cycle} with
 $L=I^{m-d}$.

First note that we have the inclusions $I^{2(m-d)}\subseteq
\overline{I^{m}}\subseteq I^{m-d+1}$: the first inclusion is due
to the hypothesis and the second one follows from Brian\c{c}on-Skoda
Theorem, see \cite{SH}. On the other hand the map
$\Tor^S_i(S/\overline{I^{m}},k)\rightarrow \Tor^S_i(S/L,k)$ factors through
the map $\Tor^S_i(S/I^{m-d+1},k)\rightarrow
\Tor^S_i(S/I^{m-d},k)$ which is zero for all $i>0$ as
$m-d+1>\rho(I)$. Hence the conditions of Lemma
\ref{cycle} hold and $\overline{I^{m}}$  is Golod.
\end{proof}
By definition for any ideal $I$ of $S$ we have $\rho(I)\geq 1$. In view of Theorem \ref{Golod}, it is interesting to
know  when $\rho(I)=1$. This is the case, at least, when $S$ is a
polynomial ring over a field of characteristic zero and $I$ is a
graded ideal of $S$.

\begin{rem}\label{derivation}
Let $S=k[x_1,\ldots,x_d]$ be a graded polynomial ring over a field
$k$ of characteristic zero and  $I\subseteq(x_1,\ldots,x_d)^2$ be a graded ideal of $S$. Set
$R=S/I$ and let $\mathcal{K}=K\otimes_SR$ be the Koszul complex of
$R$, where $K$ is the Koszul complex of $S$ with respect to
variables $x_1,\ldots,x_d$. Herzog \cite[Corollary 2]{H} gives an
explicit description of cycles of $\mathcal{K}$ in terms of the
data of the minimal free $S$-resolution of $R$. A direct consequence
of this result is that (see the proof of \cite[Theorem 1.1]{HH})
there is a subset $\mathcal{B}$ of $Z(\mathcal{K})\cap
\partial(I)\mathcal{K}$ such that the homology classes of the
elements of $\mathcal{B}$ form a $k$-basis for $H_{\geq
1}(\mathcal{K})$. Here  $\partial(I)$ denotes the ideal generated by
partial derivatives $\partial f/\partial x_i$ with $f\in I$ and
$i=1,\ldots,d$.
\end{rem}
\begin{prop}\label{graded}
Let $S$  be as  in Remark \ref{derivation}. Assume that $I\subseteq
J$ are graded ideals of $S$ such that $\partial(I)\subseteq J$. Then
the map $\Tor^{S}_i(S/I,k)\rightarrow \Tor^{S}_i(S/J,k)$, induced by
the surjection $\pi:S/I\rightarrow S/J$, is zero for all $i>0$. In
particular $\rho(I)=1$.

\end{prop}
\begin{proof}
First we  note that  the  map $\Tor^{S}_i(S/I,k)\rightarrow
\Tor^{S}_i(S/J,k)$ can be identified with the natural map
\[H_i(K\otimes \pi): H_i(K\otimes S/I)\rightarrow H_i(K\otimes
S/J).\] Thus it suffices to show that $H_i(K\otimes \pi)$ is the
zero map for all $i>0$, and it is enough to show this for the basis
elements  of $H_i(K\otimes S/I)$.  Let $h\in H_i(K\otimes S/I)$ be a
basis element. Then applying Remark \ref{derivation} $h$ can be
represented as $[z]$ for some cycle $z$ with coefficients in
$\partial(I)$. Since $\partial(I)\subseteq J$ we see that $(K\otimes
\pi)(z)$ is zero in $K\otimes S/J$. Therefore $H_i(K\otimes
\pi)([z])=0$. In particular, since $\partial(I^m)\subseteq I^{m-1}$
for any $m\geq 1$, we see that the natural map
$\Tor^{S}_i(S/I^m,k)\rightarrow \Tor^{S}_i(S/I^{m-1},k)$ is zero for
all $i>0$ and all $m\geq 1$. Then by definition $\rho(I)=1$.
\end{proof}

In light of Proposition \ref{graded}, we pose the following
question.
\begin{ques}\label{ques2} Let $S$  be a regular local ring. Is it true that $\rho(I)=1$ for any proper ideal
$I$ of $S$ or, equivalently,  the map
$$\Tor^{S}_i(S/I^m,k)\rightarrow \Tor^{S}_i(S/I^{m-1},k)$$ is zero
for all $i\geq 1$ and all $m\geq 1$?
\end{ques}
A first result in support of a positive answer to this question is
in the case of local rings of Krull dimension at most 2.

\begin{thm}\label{local}
Let $I$ be an  ideal of  a regular local ring   $S$ of    dimension at most $ 2$. Then
$\rho(I)=1$.
\end{thm}
\begin{proof}
The case of dimension $1$  is obvious. So, we assume that $S$ has
dimension 2. Let $m\geq 1$. We show that the natural map
$$\varphi_i:\Tor^{S}_i(S/I^m,k)\rightarrow \Tor^{S}_i(S/I^{m-1},k)$$
is zero for each $i=1,2$. For the case where $i=1$, the map
$\varphi_1$
 is identified with the natural map  $I^m/\n I^m\rightarrow
I^{m-1}/\n I^{m-1}$ which is
clearly zero.\\
\indent Since $S$  has dimension $2$,  for any finitely generated
$S$-module $N$  we have the natural isomorphisms
$$\Tor_2^S(N,k)\cong H_2(N\otimes K)\cong (0:_N\n).$$ Thus  $\varphi_2$ is zero if and only if
$(I^m:\n)\subseteq I^{m-1}$. In order to prove the inclusion it is
enough to show that
$(I^m:x)\subseteq I^{m-1}$ for any  $x\in \n\setminus \n^2$.\\
\indent Let $x\in \n\setminus \n^2$. Then the ring $S/xS$ is a
regular local ring of dimension one and so the image of $I$ in
$S/xS$ is generated by an element $xS+u$ for some $u\in I$. Now it
is easy to see that there are elements $r_1,\ldots, r_t$ in $S$ such
that $I=(u, r_1x,\ldots, r_tx)$. One has  $I^m\subseteq Su^m+
xI^{m-1}$. If $z\in (I^m:x)$, then we can write $zx=su^m+bx$ for
some $s\in S$ and $b\in I^{m-1}$ . Since $(x)$ is a prime ideal of
$S$ we get $u\in (x)$
 or $s\in(x)$. In both cases, one can  obtain
 that $z\in I^{m-1}$. Therefore $(I^m:x)\subseteq I^{m-1}$.
\end{proof}

Complete intersection ideals provide another piece of evidence in
support of a positive answer to Question \ref{ques2}, as the
following proposition shows. This can also be concluded from
\cite[Remark 2.12]{RSY}. But, we include a different proof here for
the sake of completeness. The particular case of this proposition is
known by \cite[Corollary 4.4]{Gv} and \cite[Theorem 6.7]{Av2} .
\begin{prop}\label{ci}
Let $I$ be an ideal of $S$ generated by a regular sequence. Then $\rho(I)=1$. In particular, $I^m$ is Golod for all $m\geq 2$.
\end{prop}
\begin{proof}
Let $ u_1,\ldots, u_p$ be a regular sequence which generates $I$.
Assume that $r$ is an integer such that  $0\leq r\leq p$. We  set
$I_r=(u_1,\ldots,u_r)$  if $r>0$ and  $I_r=(0)$ if $r=0$. To prove
the assertion  it is enough to show that for each $0\leq r\leq p$
and $m>0$ the natural map $\Tor^S_i(I_r^m,k)\rightarrow
\Tor^S_i(I_r^{m-1},k)$ is zero for all $i\geq0$. We
 do this by induction on $n=r+m$. The first step $n=1$ is obvious.
Let  $n>1$ and assume that the statement holds for all $r$ and $m$
such that $r +m \leq n-1$. We want to prove it when $n=r+m$. Since $\gr_{I_{r-1}}(S)=\oplus_{m\geq 0}I^m_{r-1}/I^{m+1}_{r-1}$ is naturally isomorphic to to the polynomial ring $S/I_{r-1}[x_1,\ldots,x_{r-1}]$, see \cite[Theorem 1.1.8]{BH},  the image of $u_r$ in $S/I_{r-1}$ is a regular sequence on $\gr_{I_{r-1}}(S)$. Using this one can see that $I^m_{r-1}=(I^m_{r-1}:u_r)$ for all $r$ and $m$.  We have then
the following short exact sequence
\begin{equation}\label{exact}
0\rightarrow I_{r-1}^m\rightarrow I_r^{m-1}\oplus
I_{r-1}^m\rightarrow u_rI_r^{m-1}+I_{r-1}^m =I_r^m\rightarrow 0
\end{equation}
where the first map is given by $a\mapsto(a,u_ra)$ and the second one is given by  $(b,a)\mapsto u_rb-a$.
 Here as usual for an ideal $J$ of $S$ we set $J^{i}=S$  if $i\leq0$.
Using the  above exact sequence, we get the following  commutative
diagram with exact rows and columns
\begin{equation}\label{diag}
\xymatrixcolsep{1.4pc}
\xymatrix{
& 0 \ar[d] & 0  \ar[d] & 0 \ar[d] & \\
0 \ar[r] &  I_{r-1}^m \ar[r] \ar@{^{(}->}[d]     & I_r^{m-1}\oplus I_{r-1}^m \ar[r] \ar@{^{(}->}[d]                                & I_r^m \ar[r] \ar@{^{(}->}[d]^q & 0\\
0 \ar[r] &  I_{r-1}^{m-1} \ar[r] \ar[d]^f & I_r^{m-2}\oplus I_{r-1}^{m-1} \ar[r] \ar[d]^h                             & I_r^{m-1} \ar[r] \ar[d]^g & 0\\
0 \ar[r] & I_{r-1}^{m-1}/I_{r-1}^m \ar[r]^-{\alpha}\ar[d] & I_r^{m-2}/I_r^{m-1}\oplus I_{r-1}^{m-1}/ I_{r-1}^m  \ar[r]^-{\beta}\ar[d] &  I_r^{m-1}/I_r^m \ar[r]\ar[d] & 0\\
& 0 & 0  & 0 &
 }
\end{equation}
where  the maps $f,g,h$ are the natural surjections and $\alpha $ is
given by   $a+I_{r-1}^{m}\mapsto (0,u_ra+ I_{r-1}^m)$ with $a\in
I_{r-1}^{m-1}$. Using this formula of $\alpha$ and the fact that the
functor $\Tor_i^S(-,k)$ is an $S$-linear functor, one can see that
 the induced map
$$ \Tor^S_i(I_{r-1}^{m-1}/I_{r-1}^m,k)
\xrightarrow{\tilde{\alpha}_i} \Tor^S_i(I_r^{m-2}/I_r^{m-1},k)\oplus
\Tor^S_i(I_{r-1}^{m-1}/ I_{r-1}^m ,k)$$ is the zero map for all $i$.
 Then  the diagram (\ref{diag}) induces the following one with exact rows and columns
\begin{equation}\label{diag 2}
\xymatrixcolsep{1.4pc}
\xymatrix{
 & \Tor^S_i(I_r^{m-1},k)\oplus \Tor^S_i( I_{r-1}^m,k) \ar[r] \ar[d]^{0}                                & \Tor^S_i(I_r^m ,k)\ar[r] \ar[d]^{\tilde{q}_i}&\Tor^S_{i-1}( I_{r-1}^m,k)\ar[d]^0\\
& \Tor^S_i(I_r^{m-2},k)\oplus \Tor^S_i(I_{r-1}^{m-1},k) \ar[r] \ar[d]^{\tilde{f_i} }                            & \Tor^S_i(I_r^{m-1},k)\ar[r]^-{\Delta_i}\ar[d]^{\tilde{h}_i}&  \Tor^S_{i-1}(I_{r-1}^{m-1},k)\ar[d]^{\tilde{g}_i} \\
0 \ar[r]& \Tor^S_i(I_r^{m-2}/I_r^{m-1},k)\oplus
\Tor^S_i(I_{r-1}^{m-1}/ I_{r-1}^m ,k) \ar[r]^-{\tilde{\beta_i}} &
\Tor^S_i(I_r^{m-1}/I_r^m,k)\ar[r]&\Tor_{i-1}(I_r^{m-1}/I_r^m,k)
 }
  \end{equation}
where the tilde maps are induced by applying the functor
$\Tor^S(-,k)$ to the diagram (\ref{diag}). The zero maps in the
columns are due to the induction hypothesis. Hence $\tilde{f_i}$ and
$\tilde{g_i}$ are injective. We want to show that $\tilde{q}_i$ is
the zero map or equivalently $\kernel \tilde{h}_i=0$. Applying \cite[Lemma 3.2]{CE} to the bottom half of the diagram, we get an exact sequence
$$\kernel \tilde{f_i}\rightarrow \kernel \tilde{h}_i\rightarrow \kernel \tilde{g}_i.$$
Since  $\tilde{f_i}$ and
$\tilde{g_i}$ are injective we get the desired conclusion. The particular case immediately follows from Theorem \ref{Golod}.
\end{proof}

\section{Golodness of Koszul ideals }  Let
$(R,\m,k)$ be a local ring  with  maximal ideal $\m$ or a standard
graded $k$-algebra with graded maximal ideal $\m$. Let $M$ be a
finitely generated $R$-module. In the graded case, $M$ is assumed to
be graded. The notion of Koszul module introduced by Herzog and
Iyengar \cite{HI}. They say that $M$ is Koszul if the linear part of
a minimal free resolution of $M$ is acyclic or equivalently if its
associated graded module $\gr_{\m}(M)=\oplus _{i\geq
0}\m^iM/\m^{i+1}M$, as a graded $\gr_{\m}(R)$-module, has a linear
resolution \cite[Proposition 1.5]{HI}. If an ideal of $R$ is Koszul
(as an $R$-module) we shall call it a  \textit{Koszul ideal}. Also,
if the residue field $k$ is Koszul we will say that the ring $R$ itself is
Koszul. Note that in the graded case $\gr_{\m}(R)$ is identified
with $R$ itself and  any graded $R$-module with linear resolution is
 Koszul. However such graded modules are not the only modules which
are Koszul see \cite[Example 1.9]{HI}. If $R$ is a graded Koszul
algebra there is a characterization of (graded) Koszul modules due
to R\"{o}mer \cite[Theorem 3.2.8]{T}: A graded $R$-module $M$ is
Koszul if and only if $M$ is componentwise linear.
\\
\indent Let $S$ be a polynomial  ring with the standard grading. It
is known  that every graded ideal of $S$ with linear resolution is
Golod \cite{AF}. This result has been generalized by  Herzog, Reiner
and Welker for  componentwise linear ideals of $S$ \cite{HRW}. Since
$S$ is a Koszul algebra, in view of the characterization of
R\"{o}mer, this can be restated in the following form:  any Koszul
ideal of $S$ is Golod.  Next we will generalize this restatement of
the result of \cite{HRW} under a weaker hypothesis
 to the local case.
 Let first recall some facts.\\
\indent Assume that $N$ is a finitely generated $S$-module. Let
$$\upsilon^S _i(N): \Tor^S_i(\n N,k)\rightarrow
\Tor^S_i(N,k)$$  be  the map which is induced by the inclusion $\n
N\subseteq N$. Then by a result of \c{S}ega \cite[Theorem
3.2(2)]{Se}, one has the following characterization of Koszul
modules.
\begin{thm}\label{Koszul 1}
An $S$-module $M$ is Koszul if and only if
$$\upsilon^S_i(\n^jM)=0 \ \ \text{for all }\ \  i\geq 0 \ \ \text{and all} \ \ j\geq 0,$$
where by convention $\n^0=S$.
  \end{thm}
 Let $S$  be a polynomial ring over a field  with the standard grading and $M$  a graded $S$-module  generated in a single degree.
 It is straightforward to see that the condition  $\upsilon^S _i(M)=0$ for all $i\geq0$,
 provides a necessary and sufficient condition for Koszulness of
 $M$. However when $M$ is not generated in  a single degree, this
 condition does not suffice for $M$ to be Koszul (i.e. componentwise linear) as
 the  following example shows:

 \begin{exam}
Let $k$ be a field of characteristic zero and $S=k[x,y,z]$. Let
$\n=(x,y,z)$, $J=(x^3,y^3,x^2z)$ and $I=J+\n^4$. The ideal $J$ does
not have a linear resolution, so $I$ is not a componentwise linear
ideal. Using a computer algebra system, the minimal free resolutions
of $I$ and $\n I$ are as follows:
$$0 \rightarrow S(-6)^7 \rightarrow S(-4)\oplus S(-5)^{15} \rightarrow S(-3)^3 \oplus S(-4)^7 \rightarrow I\rightarrow 0$$

$$
0 \rightarrow S(-6)^3 \oplus S(-7)^7\rightarrow S(-5)^9 \oplus
S(-6)^{14} \rightarrow S(-4)^8 \oplus S(-5)^6 \rightarrow
I\n\rightarrow 0$$ We claim that $\upsilon^S _i(I)=0$ for all
$i\geq0$. First note that the following conditions are equivalent:
\begin{enumerate}
\item $\upsilon^S _i(I) =0$ for all $i\geq 0$;
\item $0\rightarrow \Tor^S_i(I,k)\rightarrow
\Tor^S_i(I/\n I,k) \rightarrow \Tor^S_{i-1}(\n I,k)\rightarrow
0$ is exact for all $i\geq 0;$
\item $\beta_i(k)\beta_0(I)=\beta_i(I)+\beta_{i-1}(\n I)$ for all $i\geq
0$.
\end{enumerate}
Now comparing the Betti numbers of $I $ and $\n I$, one can see that condition $(3)$ is satisfied and then
 $\upsilon^S _i(I) =0$ for all $i\geq 0$.
\end{exam}

The following extends the result of \cite{HRW}  concerning Golodness
of componentwise linear ideals.
\begin{thm}\label{Ko-Golod}
Let $(S,\n,k)$ be a regular local ring (or a polynomial ring over
$k$) and $I\subseteq\n^2$ be an ideal  of $S$. Assume that the
natural map
$$\upsilon^S _i(I):\Tor^S_i(\n I,k)\rightarrow \Tor^S_i(I,k)$$ is
zero for all $i\geq0$. Then $I$ is Golod
\end{thm}
\begin{proof}
First note that since  $I\subseteq\n^2$, the complex $K\otimes_SS/I$ is isomorphic to the Koszul complex of $R=S/I$ with respect to a minimal system of generators of the maximal ideal of $R$.  There is a natural isomorphism $K\otimes_SS/I\cong K/I K$ of DG
algebras. We will denote by $\partial$  the differential map of the
Koszul complex $K$. The same notation will be used for the
differential map of the complex $K/I K$.

We show that $K/I K$ admits a trivial Massey
operation. Choose a set of cycles $\mathcal{C}=\{z_{\lambda}\in
K/I K\}_{\lambda\in\Lambda}$, such that
$\mathcal{B}=\{h_{\lambda}=[z_{\lambda}]\}_{\lambda\in\Lambda}$ is a
basis of $H_{\geq1}(K/I K)$. Set
$\mu(h_{\lambda})=z_{\lambda}$ and
 assume by induction that a function  $\mu:\bigsqcup_{r=1}^{p-1}
\mathcal{B}^r \rightarrow K/I K$ has been constructed for some $p
\geq 2$, and  satisfies the conditions $(M1),(M2)$ and $(M3)$.
Furthermore, if $r$ is an integer with $1\leq r\leq p-1$ and
$(\lambda_1,\ldots, \lambda_r)\in \Lambda ^r$, we assume that
$\mu(h_{\lambda_1},\ldots, h_{\lambda_r})=u_{\lambda_1,\ldots,
\lambda_r}+I K$ for some $u_{\lambda_1,\ldots, \lambda_r}\in \n K$.
Note that in the first step of the construction, since $z_{\lambda}$
is a cycle we can write $z_{\lambda}=u_{\lambda}+I K$ for some
$u_{\lambda}\in \n K$, this is due to \cite[Lemma 4.1.6(2)]{Av}. We
want to define $\mu$ on $\mathcal{B}^p$ such that the condition
$(M2)$ is satisfied. To this end it is enough to show that for each
$(\lambda_1,\ldots, \lambda_p)\in \Lambda ^p$, the element
$$z_{\lambda_1,\ldots, \lambda_p}:=\displaystyle\sum_{j=1}^{p-1}\overline{\mu(h_{\lambda_1},\ldots,h_{\lambda_j})}\mu(h_{\lambda_{j+1}},\ldots,h_{\lambda_p}),$$
which is a cycle, is indeed  a boundary. Following the above notation
we can write
$$z_{\lambda_1,\ldots, \lambda_p}=\displaystyle\sum_{j=1}^{p-1}\overline{u_{\lambda_1,\ldots,\lambda_j}}u_{\lambda_{j+1},\ldots,\lambda_p}+I K.$$
Set
$w=\displaystyle\sum_{j=1}^{p-1}\overline{u_{\lambda_1,\ldots,\lambda_j}}u_{\lambda_{j+1},\ldots,\lambda_p}$.
We establish that $\partial(w)\in  \n IK$. Using the Leibnitz rule,
we have:
\begin{align*}
\begin{split}
\partial(w)&=\displaystyle\sum_{j=1}^{p-1}\partial(\overline{u_{\lambda_1,\ldots,\lambda_j}})u_{\lambda_{j+1},\ldots,\lambda_p}+(-1)^{\mid
u_{\lambda_1,\ldots,\lambda_j}\mid}\overline{u_{\lambda_1,\ldots,\lambda_j}}
\partial(u_{\lambda_{j+1},\ldots,\lambda_p})\\
&=\displaystyle\sum_{j=1}^{p-1}\partial(\overline{u_{\lambda_1,\ldots,\lambda_j}})u_{\lambda_{j+1},\ldots,\lambda_p}-u_{\lambda_1,\ldots,\lambda_j}
\partial(u_{\lambda_{j+1},\ldots,\lambda_p})\\
&=\displaystyle\sum_{j=1}^{p-1}\bigg((a_j+\displaystyle\sum_{i=1}^{j-1}u_{\lambda_1,\ldots,\lambda_i}\overline{u_{\lambda_{i+1},\ldots,\lambda_j}}
)u_{\lambda_{j+1},\ldots,\lambda_p}-u_{\lambda_1,\ldots,\lambda_j}
(b_j+\displaystyle\sum_{s=1}^{p-j-1}
\overline{u_{\lambda_{j+1},\ldots,\lambda_{j+s}}}u_{\lambda_{j+s+1},\ldots,\lambda_p})\bigg)
\end{split}
\end{align*}
for some  $a_j, b_j\in I K$, with $1\leq j\leq p-1$. It is
straightforward to see that
$$\displaystyle\sum_{j=1}^{p-1}(\displaystyle\sum_{i=1}^{j-1}u_{\lambda_1,\ldots,\lambda_i}\overline{u_{\lambda_{i+1},\ldots,\lambda_j}} )u_{\lambda_{j+1},
\ldots,\lambda_p}-u_{\lambda_1,\ldots,\lambda_j}
(\displaystyle\sum_{s=1}^{p-j-1}
\overline{u_{\lambda_{j+1},\ldots,\lambda_{j+s}}}u_{\lambda_{j+s+1},\ldots,\lambda_p})=0.$$
Hence $\partial(w)\in  \n IK$ as claimed. Therefore $w+\n I K$ is a
cycle of $K/I\n K$. The homology class of $w+I K$ is the image of
$[w+\n I K]$ under
 the natural map $H(K/I\n K)\rightarrow
H(K/I K)$. On the other hand this map is zero by the hypothesis.
This implies that that $z_{\lambda_1,\ldots, \lambda_p}$ is a
boundary. Let $u_{\lambda_1,\ldots,\lambda_p}\in K$ such that
$\partial (u_{\lambda_1,\ldots,\lambda_p}+ IK)=z_{\lambda_1,\ldots,
\lambda_p}$. Extend $\mu$ to $\mathcal{B}^p$ by
$\mu(h_{\lambda_1},\ldots,
h_{\lambda_p})=u_{\lambda_1,\ldots,\lambda_p}+ IK$. Since by the
inductive hypothesis, $u_{\lambda_1,\ldots,\lambda_j} ,
u_{\lambda_{j+1},\ldots,\lambda_p}\in\n K$ for all $1\leq j\leq
p-1$, we see that $z_{\lambda_1,\ldots, \lambda_p}\in \n^2(K/I K)$.
Therefore $u_{\lambda_1,\ldots,\lambda_p}\in \n K$ again due to
\cite[Lemma 4.1.6(2)]{Av}. This completes the inductive step.
\end{proof}
\begin{cor}\label{Koszulideal}
Let $I$ be  a Koszul ideal of $S$. Then $\n^jI$ is a Golod ideal for
all $j>0$. Furthermore, if  $I\subseteq \n^2$, then  $I$ itself  is
Golod.
\end{cor}
\begin{proof}
Since $I$ is Koszul, we have $\upsilon^S_i(\n^jI)=0$ for all $i,j\geq
0$, by Theorem \ref{Koszul 1}. Thus the desired conclusions follow
from the above theorem.
\end{proof}
If $S$ is a polynomial ring over a field or a regular local ring containing a field, then the condition $I\subseteq \n^2$, in the above corollary is not necessary. To show this we need the following two lemmas.

\begin{lem}\label{zero}
Let $S$ be  the power series ring $k[[x_1,\ldots,x_d]]$ or the
polynomial ring $k[x_1,\ldots,x_d]$ over a field $k$.
 Assume  that $L$ is an  ideal of $k[[x_{r+1},\ldots,x_d]]$  or a graded ideal
 of $k[x_{r+1},\ldots,x_d]$ with $1\leq r \leq d$.  Then  $(x_1,\ldots,x_r)\cap LS=(x_1,\ldots,x_r)LS$.
 Moreover, $x_1,\ldots,x_r$ is a regular sequence on $S/LS$ and $LS$.
\end{lem}
\begin{proof}
We only prove the assertion in the local case. The same proof works
in the graded case. In order to simplify the notation, we denote $
x_1,\ldots,x_r$ by $\mathbf{ x}$ and we set
$Q=k[[x_{r+1},\ldots,x_d]]$. Since $\Tor^S_1(S/(\mathbf{
x}),S/LS)\cong (\mathbf{ x})\cap LS/(\mathbf{ x})LS$, we need to
show that $\Tor^S_1(S/(\mathbf{ x}),S/LS)=0$. Note that
$Q\hookrightarrow S$ is a flat extension. If $M$ is an $S$-module
and $N$ is a $Q$-module, then  there is an isomorphism
$\Tor^S_i(M,S\otimes_Q N)\cong \Tor^Q_i(M,N)$ for all $i$, see for
example \cite[Theorem 10.73]{Rot}. Since $S/(\mathbf{ x})\cong Q$,
as $Q$-modules, applying the isomorphism to $M=S/(\mathbf{ x})$ and
$N=Q/L$ we get
$$
  \Tor^S_i(S/(\mathbf{ x}),S/LS)\cong \Tor^Q_i(Q,Q/L)=0,
  $$
for all $i>0$. The sequence  $\mathbf{ x}$ is regular on an
$S$-module $M$ if the Koszul homology $H_i(\mathbf{ x},M)$ vanishes
for all $i>0$, see \cite[Corollary 1.6.19]{BH}. Since the Koszul
complex of $S$ with respect to $\mathbf{ x}$ is a  free resolution
of $S/(\mathbf{ x})$, we have $H_i(\mathbf{ x},S/LS)\cong
\Tor_i(S/\mathbf{ x},LS)=0$ for all $i>0$ and so $H_i(\mathbf{
x},LS)\cong  \Tor_i(S/\mathbf{ x},LS)=0$. Therefore $\mathbf{ x}$ is
a regular sequence on $S/LS$ and $LS$.
\end{proof}

\begin{lem}\label{subKoszul}
Let $S$  be as in the above lemma with maximal ideal  $\n$. Assume that  $I$ is a proper ideal of $S$ such that $I\nsubseteq \n^2$.
Then there is an ideal $J\subseteq \n^2$ and a part of a regular system of parameters of $S$, say $\mathbf{x}$,
such that $I=(\mathbf{x})+J$ and the initial forms of $\mathbf{x}$ in $\gr_{\n}(S)$ is a
$\gr_{\n}(J)$-regular sequence. Furthermore if $I$ is Koszul, then $J$ is Koszul.
\end{lem}
\begin{proof}
Let $r=\dim_k I/I\cap \n^2$. We can make  a change of variables and
write $I=(x_1,\ldots,x_r)+J'S$ with $J'\subseteq
k[[x_{r+1},\ldots,x_d]]\cap \n^2$ and $J'\subseteq
k[x_{r+1},\ldots,x_d]\cap \n^2$ in the polynomial case. We will show that $J:=J'S$ is the desired ideal. By
\cite[Lemma 1.1]{RV} the initial forms of $\mathbf{x}$  is a regular
sequence  on  $\gr_{\n}(J)$ if $\mathbf{x}$ is a regular sequence
on $J$ and   $(\mathbf{x})\cap \n^iJ = \mathbf{x}\n^{i-1}J $ for
all $i>0$. Since  $\mathbf{x}$ is a regular sequence on $J$ by
Lemma \ref{zero}, we only need to verify the inclusion
$(\mathbf{x})\cap \n^iJ\subseteq \mathbf{x}\n^{i-1}J $. Let
$\mathbf{x}'$ denote the sequence $x_{r+1},\ldots, x_d$ and set
$L=(\mathbf{x}')^iJ'$. We have then
\begin{align}
\begin{split}\label{valla}
(\mathbf{x})\cap \n^iJ&\subseteq(\mathbf{x})\cap\left ((\mathbf{x})\n^{i-1}J+(\mathbf{x}')^iJ\right)\\
&\subseteq (\mathbf{x})\n^{i-1}J+(\mathbf{x})\cap LS\\
&\subseteq (\mathbf{x})\n^{i-1}J+(\mathbf{x}) LS=(\mathbf{x})\n^{i-1}J,
\end{split}
\end{align}
where the second line  is due to modular law and we use  Lemma
\ref{zero} in the third line. Let now assume that $I$ is a Koszul
ideal. We want to show that $J$ is a Koszul ideal. For this by
\cite[Theorem 2.13(c)]{IT}, it is enough to show that
$J/(\mathbf{x})J$  is a Koszul module over $S$. Using the equality
$(\mathbf{x})\cap J=(\mathbf{x})J$, we get an exact sequence
$$0\rightarrow (\mathbf{x})\rightarrow I\rightarrow
J/(\mathbf{x})J\rightarrow 0.$$ Since $(\mathbf{x})$ is Koszul
applying \cite[Theorem 3.1]{Hop}, $J/(\mathbf{x})J$ is Koszul if
and only if $(\mathbf{x})\cap \n^iI=(\mathbf{x})\n^i$ for all $i>0$.
Let $i>0$. Using modular law we get the first equality below
$$(\mathbf{x})\cap
\n^iI=(\mathbf{x})\n^i+(\mathbf{x})\cap\n^iJ=(\mathbf{x})\n^i+(\mathbf{x})\n^{i-1}J=(\mathbf{x})\n^i,$$
where the second equality is obtained by applying (\ref{valla}).
\end{proof}
\begin{thm}\label{field}
Let $S$ be a polynomial ring over a field or a regular local ring containing a field. Then any proper  Koszul ideal of $S$ is Golod.
\end{thm}
\begin{proof}
First note that a local ring $R$ is Golod if and only if $\widehat{R}$, the completion of $R$ with respect to its maximal ideal, is Golod.
Also, an $R$-module $M$ is Koszul if and only if its completion $\widehat{M}$ is Koszul over $\widehat{R}$.
Therefore in the local case we can assume that $S$ is a complete regular local ring containing a field.
Then by Cohen structure theorem $S$ is a power series ring over a field. Let now $\n$ be the maximal ideal of $S$.
Assume that $I\subseteq \n$ is a Koszul ideal. If $I\subseteq\n^2$ we are done by \ref{Koszulideal}.
Otherwise by Lemma \ref{subKoszul}, there is a sequence $\mathbf{x}=x_1,\ldots,x_r$ with
$x_i\in\n\setminus\n^2$ and a Koszul ideal $J\subseteq \n^2$ such that  $I=(\mathbf{x})+J$.
Moreover $\mathbf{x}$ is a regular sequence  on $Q=S/J$. Note that $Q$ is a Golod ring by \ref{Koszulideal}. Let  $R=S/I$ and so we have $R\cong Q/\mathbf{x}Q$.  Now we can apply \cite[Proposition 5.2.4(3)]{Av} to conclude that $R$ is a Golod ring and hence by definition $I$ is a Golod ideal.

\end{proof}

\begin{prop}\label{product}
Let $I$ be an ideal of $S$ such that  $v_i(I)=0$ for all $i$. Assume
that  $L$ is a proper ideal of $S$ so that $I^2\subseteq L \subseteq
\n I$. Then $L$ is Golod. In particular, if $J$ is a proper ideal of
$S$ such that $I\subseteq J$, then $I J$ is Golod.
\end{prop}
\begin{proof}
 We have the inclusions $I^2\subseteq L\subseteq I$. Thus applying Lemma \ref{cycle}, it suffices to show that
the   map
$$\Tor^S_i(L,k)\rightarrow \Tor^S_i(I,k),$$ induced by the inclusion $L\subseteq I$, is zero for all
$i\geq0$. To this end we notice that this map factors through the
map $\upsilon^S_i(I)$ which is zero by assumption for all $i\geq0$.
Hence the claim is clear. The last assertion is clear from the first, since by the assumption  the inclusion $I^2\subseteq IJ\subseteq \n I$ holds.
\end{proof}

If $S$ is a polynomial ring over a field of characteristic zero and $I$ is a graded ideal of $S$,
Herzog and Huneke  showed that  if $I$ is strongly Golod, then  $IJ$ is (strongly) Golod for any ideal $J$ contains $I$, see \cite[Theorem 2.3(e)]{HH}. This result does not imply Corollary \ref{product} in the case where $S$ is a polynomial ring of characteristic zero, because (graded) Koszul ideals need not to be  strongly Golod
as the following simple example shows: the ideal $(xy)$ of the polynomial ring $k[x,y]$ has a linear resolution and is not strongly Golod.
 We also remark that in Corollary \ref{product} the
condition $I\subseteq J$ is necessary: let $k$ be a field and
$S=k[x,y,z,w]$ a polynomial ring. Then the graded maximal ideal
$I=(x,y,z,w)$ is obviously Koszul and  if we choose
$J=(x^2,y^2,z^2,w^2)$, then $IJ$ is not Golod, see
\cite[Example 2.1]{St}.\\

There is another characterization of Koszul modules  due to \c{S}ega
(see \cite[Theorem 2.2 (c)]{S}):
 let $(R,\m,k)$ be  a local ring ( or a standard graded
$k$-algebra). An $R$-module $M$ is Koszul if and only if the natural
map
$$\delta_{i}^j(M):\Tor^R_i(M,R/\m^{j+1})\rightarrow \Tor^R_i(M,R/\m^j)$$ is zero for all
$i>0$ and all $j\geq 0$. \\
\indent   The condition $\delta_{i}^1(M)=0$ always implies that
$\upsilon^R _i(M)=0$. In the following lemma we will prove this when
$R$ is a regular ring and $M$ is an ideal. A similar argument works
without these assumptions. Unfortunately, we do not know  whether
the converse holds.

\begin{lem}\label{Sega-char}
Let   $I\subseteq J$ be  proper  ideals of $S$. Assume that
$$\delta_{i}^1(I):\Tor^S_i(I,S/\n^2)\rightarrow \Tor^S_i(I,S/\n),$$  is
zero for all $i>0$. Then the map
$$\alpha_i:\Tor^S_i(IJ,S/\n)\rightarrow \Tor^S_i(I,S/\n),$$ induced
by the inclusion $IJ\subseteq I$, is zero for all $i\geq 0$. In
particular, $\upsilon^S_i(I)=0$ for all $i\geq 0$.
\end{lem}

\begin{proof}
Clearly $\alpha_0=0$. Let $i>0$. Using the exact sequence
$0\rightarrow\n/\n^2\rightarrow S/\n^2\rightarrow S/\n\rightarrow0$,
we get the following commutative diagram.
\begin{equation}
\xymatrix{
&\Tor^S_{i}(IJ,S/\n)\ar[d]^{\alpha_{i}}\ar[r]& \Tor^S_{i-1}(IJ,\n/\n^2)\ar[d]^{\beta_{i-1}}\\
\Tor^S_{i}(I,S/\n^2)\ar[r]^{\delta^{1}_i}&\Tor^S_{i}(I,S/\n)\ar[r]^{\gamma_i}&
\Tor^S_{i-1}(I,\n/\n^2)}
\end{equation}
 Note that  $\gamma_i$ is injective since $\delta^{1}_i=0$, by the
hypothesis. This implies that $\alpha_i=0$ if $\beta_{i-1}=0$. On
the other hand $\beta_j$ is a direct sum of $\alpha_j$ for all
$j\geq 0$. Since $\alpha_0=0$ we have $\beta_0=0$ and
 consequently $\beta_1=\alpha_1=0$.
Therefore using induction on $i$, one obtains $\alpha_i=0$. In
particular, applying the result for $J=\n$, we get $v_i^S(I)=0$.
\end{proof}
The following is an immediate consequence of the above lemma and
Theorem \ref{Ko-Golod}.
\begin{cor}
Assume that $I\subseteq \n^2$ and
$$\delta_{i}^1(I):\Tor^S_i(I,S/\n^2)\rightarrow \Tor^S_i(I,S/\n),$$  is
zero for all $i>0$. Then $I$ is Golod.
\end{cor}

 \vskip 1 cm

%%%%%%%%%%%%%%%%%%%%%%%%
\textbf{Acknowledgments}\\
The author is grateful to J\"{u}rgen Herzog, for reading the first
version of this manuscript and making  comments that improved the contents. He also wish to thank the referee  for his/her valuable suggestions
that improved  greatly the content and the presentation.

\end{document}